\documentclass[11pt]{article}
\usepackage{amssymb,amsmath}
\usepackage{latexsym,bm}
\usepackage{dsfont}

\usepackage[demo]{graphicx} % Required for inserting images
\usepackage{hyperref}
\usepackage{float}
\usepackage{indentfirst}
\usepackage{tikz}
\usepackage{amssymb}
\usepackage{amsmath,mathrsfs,amsfonts}
\usepackage{graphics}
\usepackage{amsthm}
\usepackage{setspace}
\usepackage{setspace}
\usepackage[mathscr]{eucal}
\usepackage{caption}
\usepackage{subcaption}
\usepackage{rotating}
\usepackage{enumitem}
\usepackage{cite}
\usetikzlibrary{decorations.pathreplacing}
\newtheorem{theorem}{Theorem}[section]

\newtheorem{lemma}{Lemma}[section]
\newtheorem{claim}{Claim}[section]

\newtheorem{problem}{Problem}

\catcode`@=11 \@addtoreset{equation}{section} \catcode`@=12

\newcommand{\T}{Tur$\acute{\rm a}$n~}
\newcommand{\A}{  \mathcal{A} }

\newcommand{\Sk}{\overrightarrow{S_{k,1}}}

\begin{document}
\title{Extremal oriented graphs avoiding  1-subdivision of an in-star}
\author{Zejun Huang\footnote{Corresponding author.
Email: mathzejun@gmail.com  }, \quad Chenxi Yang}
\date{School of Mathematical Sciences, Shenzhen University, Shenzhen 518060, China}
\maketitle
 \begin{abstract}
     {An oriented graph is a digraph obtained from an undirected graph by choosing an orientation for each edge.
Given a positive integer $n$ and an oriented graph $F$,
the oriented \T  number $ex_{ori}(n,F)$ is the maximum number of arcs in an $F$-free oriented graph of order $n$. In this paper, we investigate the oriented \T number  $ex_{ori}(n, \overrightarrow{S_{k,1}} )$, where $\overrightarrow{S_{k,1}}$
 is the $1$-subdivision of the in-star of order $k+1$.
We determine   $ex_{ori}(n,\overrightarrow{S_{k,1}}) $ for $k=2,3$ as well as the extremal oriented graphs.
For $k\ge 4$, we establish a  lower bound and an upper bound  on  $ex_{ori}(n,\overrightarrow{S_{k,1}})$.
 }

     \vspace{12pt}

 \end{abstract}

\section{Introduction and main results}
All undirected  graphs considered in this paper are finite and simple.
Let $G$   be an undirected   graph. We denote $V(G)$ and $E(G)$  the {\it vertex set} and the {\it edge  set} of  $G$, respectively. We call $|V(G)|$ the {\it order} and $e(G)=|E(G)|$ the {\it size} of $G$.
Similarly, given a digraph $D$, we denote by $V(D)$, $A(D)$, $a(D)$  the {\it vertex set}, the {\it  arc   set} and the {\it size} of $D$, respectively. We denote by $K_r$, $\overrightarrow{K_r}$ and $\overrightarrow{C_r}$  the complete undirected graph, the complete digraph and the directed cycle of order $r$, respectively.

Suppose $D$ is a digraph with  $u, v\in V(D)$. If $uv\in A(D)$, then $u$ is called an {\it in-neighbor} of $v$ and $v$ is called an {\it out-neighbor} of $u$. The set of in-neighbors (out-neighbors) of $u$ is called its {\it in-neighborhood} ({\it out-neighborhood}), whose cardinality is called the {\it in-degree} ({\it out-degree}) of $u$. If all vertices of  $D$ have the same in-degree $k$, we say $D$ is an {\it in-degree $k$-regular digraph}. When $X$ is a vertex set or an arc set of a digraph (or undirected graph) $D$, we use $D[X]$ to denote the subgraph of $D$ induced by $X$.

An oriented graph is a digraph obtained from an undirected graph by choosing an orientation for each edge.  The {\it $1$-subdivision} of a digraph $D$ is the digraph obtained from $D$ by replacing each arc $uv\in \A(D)$  with a directed 2-path $uwv$ along with a new vertex $w$. Denote by  $\overrightarrow{S_{k}}$ the in-star of order $k+1$ and  $\overrightarrow{S_{k,1}}$ the $1$-subdivision of  $\overrightarrow{S_{k}}$. The {\it center} of  $\overrightarrow{S_{k,1}}$
is the same as the center of $\overrightarrow{S_{k}}$; see Figure 1 and Figure 2 for the diagrams of $\overrightarrow{S_{k}}$ and $\overrightarrow{S_{k,1}}$.

\begin{figure}[h]
    \centering
    \begin{minipage}{0.5\textwidth}
    \centering
    \begin{tikzpicture}
       \draw[decorate,decoration={brace,mirror}](-0.25,1)--(-0.25,-1);
       \filldraw (0,1) circle (0.1);
       \filldraw (0,0) circle (0.1);
       \filldraw (0,-1) circle (0.1);
       \filldraw (1,0) circle (0.1);
       \draw[-latex] (0,1)--(0.9,0.1);
       \draw[-latex] (0,0)--(0.9,0);
       \draw[-latex] (0,-1)--(0.9,-0.1);
       \node [left] at (-0.25,0) {$k$};
       \node [above] at (1,0) {$v$};
       \node at (0,-0.5) {$\vdots$};
    \end{tikzpicture}
    \caption{$\protect\overrightarrow{S_k}$ centered at $v$}
    \label{fig1}
    \end{minipage}%
    \begin{minipage}{0.5\textwidth}
    \centering
    \begin{tikzpicture}
       \draw[decorate,decoration={brace,mirror}](-1.25,1)--(-1.25,-1);
       \filldraw (-1,1) circle (0.1);
       \filldraw (-1,0) circle (0.1);
       \filldraw (-1,-1) circle (0.1);
       \filldraw (0,1) circle (0.1);
       \filldraw (0,0) circle (0.1);
       \filldraw (0,-1) circle (0.1);
       \filldraw (1,0) circle (0.1);
       \draw[-latex] (-1,1)--(-0.1,1);
       \draw[-latex] (-1,0)--(-0.1,0);
       \draw[-latex] (-1,-1)--(-0.1,-1);
       \draw[-latex] (0,1)--(0.9,0.1);
       \draw[-latex] (0,0)--(0.9,0);
       \draw[-latex] (0,-1)--(0.9,-0.1);
       \node [left] at (-1.25,0) {$k$};
       \node [above] at (1,0) {$v$};
       \node at (-0.5,-0.5) {$\vdots$};
    \end{tikzpicture}
    \caption{$\protect\overrightarrow{S_{k,1}}$ centered at $v$}
    \label{fig2}
    \end{minipage}
\end{figure}

Two  graphs $G_1 $ and $G_2 $ are {\it  isomorphic}, written $G_1\cong G_2$, if there is a bijection $\sigma: V(G_1)\rightarrow V(G_2)$ such that $uv\in E(G_1)$ if and only if $\sigma(u)\sigma(v)\in E(G_2)$. For  graphs $G$ and $H$, we say $G$ contains a {\it copy} of $H$ if $G$ contains a subgraph isomorphic to $H$. If $G$ contains no copy of $H$, then $G$ is said to be {\it $H$-free}. The same conceptions can be defined for digraphs.

  \T problems on undirected graphs  consider the maximum size of $F$-free graphs of  order $n$ with $F$ being a single graph  or a family of graphs, which have been attracting  a lot of studies since Mantel \cite{WM} and \T \cite{PT} determined the maximum size of triangle-free and $K_{t}$-free graphs of order $n$; see \cite{BB,VN}.  The studies of \T problems on digraphs  were initiated by Brown and Harary \cite{BH}, who determined the maximum sizes of  $\overrightarrow{K_r}$-free digraphs and $\overrightarrow{T_r}$-free digraphs as well as the extremal digraphs, where $\overrightarrow{T_r}$ is a tournament of order $r$. Brown, Erd\H{o}s and  Simonovits \cite{BES,BES2,BES3} presented some asymptotic results on extremal digraphs avoiding a family of digraphs. H$\ddot{\rm a}$ggkvist and Thomassen \cite{HT}, Zhou and Li \cite{ZL} characterized the extremal $\overrightarrow{C_k}$-free digraphs.  Bermond,   Germa,   Heydemann and Sotteau \cite{BGHS} determined the maximum size of digraphs avoiding all directed cycles of length at most $r$; Chen and Chang \cite{CC,CC2} studied the same problem by posting additional condition on the minimum in-degree and out-degree.  The authors of \cite{HL1,HL2,HL3,HL4,HZ,ZL1,ZL2,ZL3,ZL4,TES,HW} studies extremal digraphs avoiding certain distinct walks or paths with the same initial vertex and terminal vertex.

  \T problems on digraphs when the host digraphs are oriented graphs are also interesting.    Howalla, Dabboucy and Tout \cite{HDT, HDT2} determined the
maximum size of  oriented digraphs avoiding $k$ directed paths with the same initial vertex and terminal vertex
for $k  =1, 2, 3$. Maurer, Rabinovitch and Trotter \cite{MRT} studied the maximum size of  subgraphs of a transitive tournament which contain at most one directed path from $x$ to $y$ for any two
distinct vertices $x, y$.

  Given a positive integer $n$ and a digraph $F$,
the {\it oriented Tur$\acute{a}$n number} $ex_{ori}(n,F)$ is the maximum size of $F$-free oriented graphs of order  $n$. We denoted by $EX_{ori}(n,F)$
the extremal $F$-free oriented graphs of order $n$ that attain  the size $ex_{ori}(n,F)$.

We are interested in the following problem.

\begin{problem}
Let $n,k$ be positive integers. Determine $ex_{ori}(n,\Sk)$ and $EX_{ori}(n,\Sk)$.
\end{problem}
In this paper, we solve Problem 1 for the cases $k=2$ and $k=3$. For $k\ge 4$, we  establish bounds on  $ex_{ori}(n,\Sk)$.
Our main results states as follows.

\begin{theorem}\label{th1.1}
 Let  $n\geq 16$ be an integer. Then
  \begin{equation}
  \nonumber
  ex_{ori}(n,\overrightarrow{S_{2,1}})=\left\lfloor\frac{(n+1)^2}{4}\right\rfloor.
\end{equation}
Moreover, $D\in EX_{ori}(n,\overrightarrow{S_{2,1}})$ if and only if $V(D)$ has a  partition $V(D)=X\cup Y$ such that
\begin{itemize}
\item[(i)]$|X|\in \{\lfloor(n-1)/2\rfloor,\lceil (n-1)/2\rceil\}$;
 \item[(ii)] $xy\in A(D)$ for all $x\in X,y\in Y$;
\item[(iii)]$D[X]$ is an empty digraph, $D[Y]$ is an in-degree $1$-regular digraph.
\end{itemize}
\end{theorem}

\begin{theorem}\label{th1.2}
  Let  $n\geq 40$ be an integer. Then
  \begin{equation}
  \nonumber
  ex_{ori}(n,\overrightarrow{S_{3,1}})=\left\lfloor\frac{(n+2)^2}{4}\right\rfloor.
\end{equation}
Moreover, $D\in EX_{ori}(n,\overrightarrow{S_{3,1}})$ if and only if $V(D)$ has a  partition $V(D)=X\cup Y$ such that
\begin{itemize}
\item[(i)]$|X|\in \left\{\lfloor n/2\rfloor-1, \lceil n/2\rceil-1\right\}$;
 \item[(ii)] $xy\in A(D)$ for all $x\in X,y\in Y$;
\item[(iii)]$D[X]$ is an empty digraph, $D[Y]$ is an in-degree $2$-regular digraph.
\end{itemize}
\end{theorem}

\begin{theorem}\label{th1.3}
Let $n, k$ be positive integers with $k\ge 4$ and $n\ge 3k+1$. Then
  $$\left\lfloor \frac{(n+k-1)^2}{4}\right\rfloor \leq ex_{ori}(n,\overrightarrow{S_{k,1}})\leq \left\lfloor\frac{(n+k-1)^2}{4}\right\rfloor+(k-1)n.$$
\end{theorem}
We will present some preliminaries in   Section 2 and prove the above theorems in Section 3.
\section{Preliminaries}

Firstly we present some definitions and notations.
Given a vector $\mathbf{x}=(x_1,x_2,\ldots,x_n)$, we rearrange its components in   decreasing order as as $x_{[1]}\ge x_{[2]}\ge \cdots\ge x_{[n]}$. Given two  nonnegative  vectors $\mathbf{x}=(x_1,x_2,\ldots ,x_s)\in \mathbb{R}^{s}$, $\mathbf{y}=(y_1,y_2,\ldots ,y_t)\in \mathbb{R}^t$, if $s\geq t$ and $x_{[i]}\geq y_{[i]}$ for all $1\leq i\leq  t$, then we say $\mathbf{x}$ {\it covers} $\mathbf{y}$,
  denoted  $\mathbf{x} \rhd \mathbf{y}$. We denote by $\mathbf{x}\not\rhd \mathbf{y}$ if $\mathbf{x}$ does not cover $\mathbf{y}$.

The {\it underlying graph} of a digraph $D$ is  the undirected graph obtained by replacing each arc  by an edge with the same ends. A set of vertex disjoint edges (arcs) $ \{u_1v_1,\ldots,u_kv_k\}$ in an undirected graph (digraph) is called a {\it $k$-matching}. The {\it matching number} of an undirected graph (digraph)     is the  maximum cardinality of  matchings in it.

  Let $D$ be a digraph with $E\subseteq A(D)$, and let  $X, Y$ be disjoint subsets of $ V(D)$ with $Y=\{y_1,y_2,\ldots,y_k\}$. Suppose $G$ is the underlying graph of $D$. We will need the following notations:

\begin{itemize}
\item[]\vskip -0.4cm$N_X^-(u)$:  the in-neighborhood    of $u$  from $X$, which is also denoted by $N^-(u)$ when $X=V(D)$;
\item[]\vskip -0.3cm$N_X^+(u)$:    the out-neighborhood      of $u$  from $X$, which is also denoted by $N^+(u)$ when $X=V(D)$;
\item[]\vskip -0.3cm$d_X^-(u)$:  the cardinality of $N_X^-(u)$, which is also denoted by $d^-(u)$ when $X=V(D)$;
\item[]\vskip -0.3cm$d_X^+(u)$:  the cardinality of $N_X^+(u)$, which is also denoted by $d^+(u)$ when $X=V(D)$;
\item[]\vskip -0.3cm$\Delta^-(D)$: the maximum in-degree     of $D$, which is also denoted by $\Delta^-$ when $D$ is clear from the context;
\item[]\vskip -0.3cm$d_X^-(Y)$: $(d^-_X(y_1),d^-_X(y_2),\ldots,d^-_X(y_k))$, which is denoted by $d^-(Y)$ when $X=V(D)$ and $\widetilde{d}^-(Y)$
\item[]\vskip -0.3cm\quad\quad\quad~~when $X=V(D)-Y$;
\item[]\vskip -0.3cm $[X,Y]$: the set of arcs from  $X$ to $Y$,  which also represents the set of edges between   $X$ and $Y$
\item[]\vskip -0.3cm\quad\quad\quad~~in $G$;
\item[]\vskip -0.3cm $a(X,Y)$: the cardinality of $[X,Y]$, which is abbreviated as $a(X)$ when $X=Y$;
\item[]\vskip -0.3cm  $\bar{X}$: the complement of $X$ in $V(D)$, i.e., $V(D)-X$.
\item[]\vskip -0.3cm$N^-(X)$: the set of all in-neighbors of the vertices of $X$ in $D$;
\item[]\vskip -0.3cm$N_G(X)$: the set of all neighbors of the vertices of $X$ in $G$;
\end{itemize}

For convenience, sometimes we simply write $v$ for a singleton $\{v\}$. If a bipartite graph $G$ has a vertex partition $V(G)=X\cup Y$ such that $G[X]$ and $G[Y]$ are empty graphs, then we say $G$ is a $(X,Y)$-bigraph, denoted $G(X,Y)$.

 We will need the following lemmas.

\begin{lemma}\cite{DW}\label{lemma2.1}
  An $(X,Y)$-bigraph $G$ has a matching that saturates $X$ if and only if $|N(S)|\geq |S|$ for all $S\subseteq X$.
\end{lemma}

\begin{lemma}\label{lemma2.2}
Let $D$ be an oriented graph and $v\in V(D)$  with   $\{u_1,u_2,\ldots,u_k\}\subseteq N^-(v)$, $k\ge 2$. If
$\widetilde{d}^-(u_1,u_2,\ldots,u_k)\rhd(k,k-1,\ldots,1)$,
then $D$ contains an $\overrightarrow{S_{k,1}}$ centered at $v$.
\end{lemma}

\begin{proof}
  Denote by $Y=\{u_1,u_2,\ldots,u_k\}$ and $E=[\bar{Y}\setminus\{v\},Y]$. Suppose $G$ is the underlying graph of $D[E]$. Then $V(G)$ has a bipartition $V(G)=\tilde{Y}\cup Y'$ with  $\tilde{Y}\subseteq \bar{Y}\setminus \{v\}$ and $Y'\subseteq Y$ being
 independent sets in $G$.
  Since $\widetilde{d}^-(u_1,u_2,\ldots,u_k)\rhd (k,k-1,\ldots,1)$,  we have $Y'=Y$ and  $|N_G(X)|\geq |X|$ for all $ X\subseteq Y$.
  Applying Lemma \ref{lemma2.1} on $G(\tilde{Y},Y)$, $G(\tilde{Y},Y)$ has a $k$-matching saturated $Y$. It follows that
  $D$ contains an $\overrightarrow{S_{k,1}}$ centered at $v$.
\end{proof}

\begin{lemma}\label{lemma2.3}
  Let $D$ be an oriented graph and $v\in V(D)$ with $N^-(v)=\{u_1,u_2,\ldots, u_k\}$.
\begin{itemize}
 \item[(i)] If
   $d^-(u_1,u_2,\ldots ,u_k)\rhd (4,3,3)$, then
  $D$ contains an $\overrightarrow{S_{3,1}}$ centered at $v$.
 \item[(ii)] If
   $d^-(u_1,u_2,\ldots ,u_k)\rhd (2,2)$, then
  $D$ contains an $\overrightarrow{S_{2,1}}$ centered at $v$.
  \end{itemize}
\end{lemma}

\begin{proof}
Without loss of generality, we assume $u_1\ge u_2\ge \cdots\ge u_k$.

 (i)~~ Denote by $S=\{u_1,u_2,u_3\}$. Suppose $d^-(u_i)=d_i$ for $i=1,2,3$.
  Since $D$ is an oriented graph, we have
\begin{equation}\label{eq2.1}
d_{S}^-(u_i)\leq 2 \quad{\rm  for \quad} i=1,2,3
 \end{equation}
 and
 \begin{equation}\label{eq2.2}
 d_{S}^-(u_1)+d_{S}^-(u_2)+d_{S}^-(u_3)\leq 3.
 \end{equation}
  Since  $d^-(u_1,u_2,\ldots ,u_k)\rhd (4,3,3)$, we have $\widetilde{d}^-(u_1,u_2,u_3)\rhd (3,2,1)$.
 Applying Lemma \ref{lemma2.2} we get the conclusion.

 (ii)~~The proof of (ii) is similar.
\end{proof}

\begin{lemma}\label{lemma2.4}
  Let $D$ be an $\overrightarrow{S_{3,1}}$-free oriented graph and $v\in V(D)$ with $d^-(v)\geq 3$. Suppose $S=\{u_1,u_2,u_3\}\subseteq N^-(v)$.
\begin{itemize}
 \item[(i)] If  $\widetilde{d}^-(S)=(2,2,2)$,
  then the subgraph of $D$ induced by $\{v,u_1,u_2,u_3\}\cup N^-(S)$ contains a spanning subgraph isomorphic to $H_1$ (Figure \ref{fig3}).
\item[(ii)]If  $d^-(S)=(3,3,3)$,
  then the subgraph of $D$ induced by $\{v,u_1,u_2,u_3\}\cup N^-(S)$  contains a spanning subgraph isomorphic to $H_2$ (Figure \ref{fig4}).
  \end{itemize}
\end{lemma}
\begin{figure}[h]
    \centering
    \begin{minipage}{0.5\textwidth}
    \centering
    \begin{tikzpicture}
      \filldraw(0,1) circle(0.1);
      \filldraw(0,0) circle(0.1);
      \filldraw(0,-1) circle(0.1);
      \filldraw(1,0) circle(0.1);
      \filldraw(-1,0.5) circle(0.1);
      \filldraw(-1,-0.5) circle(0.1);
      \node [above] at (1,0) {$v$};
      \node [above] at (0,0) {$u_1$};
      \node [above] at (0,1) {$u_2$};
      \node [below] at (0,-1) {$u_3$};
      \draw[-latex](0,1)--(0.9,0.1);
      \draw[-latex](0,0)--(0.9,0);
      \draw[-latex](0,-1)--(0.9,-0.1);
      \draw[-latex](-1,0.5)--(-0.1,1);
      \draw[-latex](-1,0.5)--(-0.1,0);
      \draw[-latex](-1,0.5)--(-0.1,-1);
      \draw[-latex](-1,-0.5)--(-0.1,1);
      \draw[-latex](-1,-0.5)--(-0.1,0);
      \draw[-latex](-1,-0.5)--(-0.1,-1);
    \end{tikzpicture}
    \caption{$H_1$}
    \label{fig3}
    \end{minipage}%
    \begin{minipage}{0.5\textwidth}
    \centering
    \begin{tikzpicture}
      \filldraw(0,1) circle(0.1);
      \filldraw(0.5,0) circle(0.1);
      \filldraw(0,-1) circle(0.1);
      \filldraw(1.5,0) circle(0.1);
      \filldraw(-1,0.5) circle(0.1);
      \filldraw(-1,-0.5) circle(0.1);
      \node [above] at (1.5,0) {$v$};
      \node [above] at (0,1) {$u_1$};
      \node [above] at (0.6,0) {$u_2$};
      \node [below] at (0,-1) {$u_3$};
      \draw[-latex](0,1)--(1.4,0.1);
      \draw[-latex](0.5,0)--(1.4,0);
      \draw[-latex](0,-1)--(1.4,-0.1);
      \draw[-latex](-1,0.5)--(-0.1,1);
      \draw[-latex](-1,0.5)--(0.4,0);
      \draw[-latex](-1,0.5)--(-0.1,-1);
      \draw[-latex](-1,-0.5)--(-0.1,1);
      \draw[-latex](-1,-0.5)--(0.4,0);
      \draw[-latex](-1,-0.5)--(-0.1,-1);
      \draw[-latex](0,1)--(0.4,0.1);
      \draw[-latex](0.5,0)--(0.1,-0.9);
      \draw[-latex](0,-1)--(0,0.9);
    \end{tikzpicture}
    \caption{$H_2$}
    \label{fig4}
    \end{minipage}
\end{figure}
\begin{proof}
 (i) Suppose $E=[\bar{S},S]$ and $G$ is the underlying graph of $D[E]$.
  Since $D$ is $\overrightarrow{S_{3,1}}$-free, $G$ has no  matching saturating $S$.

  Since $\widetilde{d}^-(u_1,u_2,u_3)=(2,2,2)$, we have  $$|N_G(u_1)|=|N_G(u_2)|=|N_G(u_3)|=2.$$
  Since we have $|N_G(S')|=2\geq |S'|$ for all  $ S'\subseteq S$ with $|S'|\leq 2$,
 applying Lemma \ref{lemma2.1}, we have $|N_G(S)|=2$, which leads to $|N^-(S)|=2$.
 Therefore, the subgraph of $D$ induced by $\{v,u_1,u_2,u_3\}\cup N^-(S)$ contains a spanning subgraph isomorphic to $H_1$.

  (ii)~  Since $D$ is an oriented graph and $|S|=3$,
  we have \eqref{eq2.1} and \eqref{eq2.2}. On the other hand,
  since $D$ is $\overrightarrow{S_{3,1}}$-free, applying Lemma  \ref{lemma2.2},
   we have $$\widetilde{d}^-(u_1,u_2,u_3)\not\rhd (3,2,1).$$ Therefore,
  $d^-(u_1)=d^-(u_2)=d^-(u_3)=3$ leads to $d_{S}^-(u_1)=d_{S}^-(u_2)=d_{S}^-(u_3)=1$  and
  $\widetilde{d}^-(S)=(2,2,2)$. By (i), the subgraph $D_1$ of $D$ induced by $\{v,u_1,u_2,u_3\}\cup N^-(S)$ contains a spanning subgraph isomorphic to $H_1$ and a triangle with vertices  $u_1,u_2,u_3$. Hence, $D_1$ contains a spanning subgraph isomorphic to $H_2$.
\end{proof}

\begin{lemma}\label{lemma2.5}
  Let $D$ be an $\overrightarrow{S_{3,1}}$-free oriented graph and $v\in V(D)$. If $d^-(N^-(v))\rhd (3,3,3)$,
  then   $$d^-(N^-(v))=(3,3,3,1,0,0,\ldots ,0) \quad or\quad
 d^-(N^-(v))=(3,3,3,0,0,\ldots ,0).$$
\end{lemma}

\begin{proof}
Suppose $N^-(v)=\{u_1,\ldots,u_k\}$. Without loss of generality, we may assume $d^-(u_1)\ge d^-(u_2)\ge \cdots\ge d^-(u_k)$.
 Applying Lemma \ref{lemma2.3}, we have  $d^-(u_1,u_2,u_3)\not\rhd (4,3,3).$ Since $\ d^-(u_i)\geq 3$ for $ i=1,2,3$, we get
 \begin{equation}\label{eq2.3}
 d^-(u_1)=d^-(u_2)=d^-(u_3)=3.
 \end{equation}
  Applying Lemma \ref{lemma2.4},  we have $|N^-(u_1,u_2,u_3)|=2$ and   $D[V_1]$  contains a spanning subgraph $D_1$ isomorphic to $H_2$, where   $V_1=\{v,u_1,u_2,u_3\}\cup N^-(u_1,u_2,u_3)$. Without loss of generality, we assume $D_1=H_3$ with $N^-(u_1,u_2,u_3)=\{w_1,w_2\}$; see Figure \ref{fig5}.
Denote by $V_2= \{ u_1,u_2,u_3, w_1,w_2\}$.

Firstly we claim
\begin{equation}\label{eq2.4}
d^-(u)=0 \quad {\rm for~~ all}\quad  u\in N^-(v)\setminus V_2.
\end{equation}
  Otherwise,  suppose $u_4\in N^-(v)\setminus V_2\ne \emptyset$ with $d^-(u_4)\ne0$ and $w_3\in N^-(u_4)$. Notice that $w_3\ne v$.
If $w_3\notin \{u_1,u_2,u_3,w_1,w_2\}$, then $D$ has a subgraph $H_4$; see   Figure \ref{fig6}.
  If $w_3\in \{u_1,u_2,u_3\}$, without loss of generality, we may assume $w_3=u_3$. Then $D$ has a subgraph $H_5$; see   Figure \ref{fig7}.
  If $w_3\in \{w_1,w_2\}$,  without loss of generality, we may assume  $w_3=w_2$. Then $D$ has a subgraph $H_6$; see   Figure \ref{fig8}. Since  $H_i$  contains an $\overrightarrow{S_{3,1}}$ centered at $v$ for $i=4,5,6$, we   get a contradiction.

Combining \eqref{eq2.3} and \eqref{eq2.4}, now we only need to prove that at most one vertex in $N^-(v)\cap\{w_1,w_2\}$ has nonzero in-degree,
which equals one.

 Suppose $w_i\in N^-(v)$ with $i\in \{1,2\}$.
  If $D$ has a vertex $w_3\in N^-(w_i)$ such that $w_3\neq w_{3-i}$, then $D$ has a subgraph $H_7$ (Figure \ref{fig9}), which contains an $\overrightarrow{S_{3,1}}$, a contradiction. Therefore, we have $N^-(w_i)=\{w_{3-i}\}$ or $N^-(w_i)=\emptyset$.
Since $D$ is an oriented graph, $D$ contains at most one vertex in $N^-(v)\cap\{w_1,w_2\}$ with nonzero in-degree, which equals one.
\end{proof}

\begin{figure}[h]
    \centering
    \begin{minipage}[t]{0.33\textwidth}
    \centering
    \begin{tikzpicture}
      \filldraw(0,1) circle(0.1);
      \filldraw(0.5,0) circle(0.1);
      \filldraw(0,-1) circle(0.1);
      \filldraw(1.5,0) circle(0.1);
      \filldraw(-1,0.5) circle(0.1);
      \filldraw(-1,-0.5) circle(0.1);
      \draw[-latex](0,1)--(1.4,0.1);
      \draw[-latex](0.5,0)--(1.4,0);
      \draw[-latex](0,-1)--(1.4,-0.1);
      \draw[-latex](-1,0.5)--(-0.1,1);
      \draw[-latex](-1,0.5)--(0.4,0);
      \draw[-latex](-1,0.5)--(-0.1,-1);
      \draw[-latex](-1,-0.5)--(-0.1,1);
      \draw[-latex](-1,-0.5)--(0.4,0);
      \draw[-latex](-1,-0.5)--(-0.1,-1);
      \draw[-latex](0,1)--(0.4,0.1);
      \draw[-latex](0.5,0)--(0.1,-0.9);
      \draw[-latex](0,-1)--(0,0.9);
      \node[right] at (1.5,0) {$v$};
      \node[above] at (0,1) {$u_1$};
      \node[above] at (0.6,0) {$u_2$};
      \node[below] at (0,-1) {$u_3$};
      \node[left] at (-1,0.5) {$w_1$};
      \node[left] at (-1,-0.5) {$w_2$};
    \end{tikzpicture}
    \caption{$H_3$}
    \label{fig5}
    \end{minipage}%
    \begin{minipage}[t]{0.33\textwidth}
    \centering
    \begin{tikzpicture}
      \filldraw(0,1) circle(0.1);
      \filldraw(0.5,0) circle(0.1);
      \filldraw(0,-1) circle(0.1);
      \filldraw(0,-2) circle(0.1);
      \filldraw(-1,-2) circle(0.1);
      \filldraw(1.5,0) circle(0.1);
      \filldraw(-1,0.5) circle(0.1);
      \filldraw(-1,-0.5) circle(0.1);
      \draw[-latex](0,1)--(1.4,0.1);
      \draw[-latex](0.5,0)--(1.4,0);
      \draw[-latex](0,-1)--(1.4,-0.1);
      \draw[-latex](-1,0.5)--(-0.1,1);
      \draw[-latex](-1,0.5)--(0.4,0);
      \draw[-latex](-1,0.5)--(-0.1,-1);
      \draw[-latex](-1,-0.5)--(-0.1,1);
      \draw[-latex](-1,-0.5)--(0.4,0);
      \draw[-latex](-1,-0.5)--(-0.1,-1);
      \draw[-latex](0,1)--(0.4,0.1);
      \draw[-latex](0.5,0)--(0.1,-0.9);
      \draw[-latex](0,-1)--(0,0.9);
      \draw[-latex](0,-2)--(1.4,-0.1);
      \draw[-latex](-1,-2)--(-0.1,-2);
      \node[right] at (1.5,0) {$v$};
      \node[above] at (0,1) {$u_1$};
      \node[above] at (0.6,0) {$u_2$};
      \node[below] at (0,-1) {$u_3$};
      \node[left] at (-1,0.5) {$w_1$};
      \node[left] at (-1,-0.5) {$w_2$};
      \node[below] at (-1,-2) {$w_3$};
      \node[below] at (0,-2) {$u_4$};
    \end{tikzpicture}
    \caption{$H_4$}
    \label{fig6}
    \end{minipage}%
    \begin{minipage}[t]{0.33\textwidth}
    \centering
    \begin{tikzpicture}
      \filldraw(0,1) circle(0.1);
      \filldraw(0.5,0) circle(0.1);
      \filldraw(0,-1) circle(0.1);
      \filldraw(0,-2) circle(0.1);
      \filldraw(1.5,0) circle(0.1);
      \filldraw(-1,0.5) circle(0.1);
      \filldraw(-1,-0.5) circle(0.1);
      \draw[-latex](0,1)--(1.4,0.1);
      \draw[-latex](0.5,0)--(1.4,0);
      \draw[-latex](0,-1)--(1.4,-0.1);
      \draw[-latex](-1,0.5)--(-0.1,1);
      \draw[-latex](-1,0.5)--(0.4,0);
      \draw[-latex](-1,0.5)--(-0.1,-1);
      \draw[-latex](-1,-0.5)--(-0.1,1);
      \draw[-latex](-1,-0.5)--(0.4,0);
      \draw[-latex](-1,-0.5)--(-0.1,-1);
      \draw[-latex](0,1)--(0.4,0.1);
      \draw[-latex](0.5,0)--(0.1,-0.9);
      \draw[-latex](0,-1)--(0,0.9);
      \draw[-latex](0,-2)--(1.4,-0.1);
      \draw[-latex](0,-1)--(-0,-1.9);
      \node[right] at (1.5,0) {$v$};
      \node[above] at (0,1) {$u_1$};
      \node[above] at (0.6,0) {$u_2$};
      \node[left] at (0,-1) {$u_3(w_3)$};
      \node[left] at (-1,0.5) {$w_1$};
      \node[left] at (-1,-0.5) {$w_2$};
      \node[left] at (0,-2) {$u_4$};
    \end{tikzpicture}
    \caption{$H_5$}
    \label{fig7}
    \end{minipage}%
\end{figure}

\begin{figure}[h]
    \centering
    \begin{minipage}[t]{0.33\textwidth}
    \centering
    \begin{tikzpicture}
      \filldraw(0,1) circle(0.1);
      \filldraw(0.5,0) circle(0.1);
      \filldraw(0,-1) circle(0.1);
      \filldraw(0,-2) circle(0.1);
      \filldraw(1.5,0) circle(0.1);
      \filldraw(-1,0.5) circle(0.1);
      \filldraw(-1,-0.5) circle(0.1);
      \draw[-latex](0,1)--(1.4,0.1);
      \draw[-latex](0.5,0)--(1.4,0);
      \draw[-latex](0,-1)--(1.4,-0.1);
      \draw[-latex](-1,0.5)--(-0.1,1);
      \draw[-latex](-1,0.5)--(0.4,0);
      \draw[-latex](-1,0.5)--(-0.1,-1);
      \draw[-latex](-1,-0.5)--(-0.1,1);
      \draw[-latex](-1,-0.5)--(0.4,0);
      \draw[-latex](-1,-0.5)--(-0.1,-1);
      \draw[-latex](0,1)--(0.4,0.1);
      \draw[-latex](0.5,0)--(0.1,-0.9);
      \draw[-latex](0,-1)--(0,0.9);
      \draw[-latex](0,-2)--(1.4,-0.1);
      \draw[-latex](-1,-0.5)--(-0.1,-1.9);
      \node[right] at (1.5,0) {$v$};
      \node[above] at (0,1) {$u_1$};
      \node[above] at (0.6,0) {$u_2$};
      \node[below] at (0,-1) {$u_3$};
      \node[left] at (-1,0.5) {$w_1$};
      \node[left] at (-1,-0.5) {$w_2(w_3)$};
      \node[below] at (0,-2) {$u_4$};
    \end{tikzpicture}
    \caption{$H_6$}
    \label{fig8}
    \end{minipage}
    \begin{minipage}[t]{0.33\textwidth}
    \centering
    \begin{tikzpicture}
      \filldraw(0,1) circle(0.1);
      \filldraw(0.5,0) circle(0.1);
      \filldraw(0,-1) circle(0.1);
      \filldraw(1.5,0) circle(0.1);
      \filldraw(-1,0.5) circle(0.1);
      \filldraw(-1,-0.5) circle(0.1);
      \filldraw(-2,0.5) circle(0.1);
      \draw[-latex](0,1)--(1.4,0.1);
      \draw[-latex](0.5,0)--(1.4,0);
      \draw[-latex](0,-1)--(1.4,-0.1);
      \draw[-latex](-1,0.5)--(-0.1,1);
      \draw[-latex](-1,0.5)--(0.4,0);
      \draw[-latex](-1,0.5)--(-0.1,-1);
      \draw[-latex](-1,-0.5)--(-0.1,1);
      \draw[-latex](-1,-0.5)--(0.4,0);
      \draw[-latex](-1,-0.5)--(-0.1,-1);
      \draw[-latex](0,1)--(0.4,0.1);
      \draw[-latex](0.5,0)--(0.1,-0.9);
      \draw[-latex](0,-1)--(0,0.9);
      \draw[-latex](-2,0.5)--(-1.1,0.5);
      \draw[-latex](-1,0.5)--(1.4,0.1);
      \node[right] at (1.5,0) {$v$};
      \node[above] at (0,1) {$u_1$};
      \node[below] at (0.6,0) {$u_2$};
      \node[below] at (0,-1) {$u_3$};
      \node[above] at (-1,0.5) {$w_i$};
      \node[left] at (-1,-0.5) {$w_{3-i}$};
      \node[left] at (-2,0.5) {$w_{3}$};
    \end{tikzpicture}
    \caption{$H_7$}
    \label{fig9}
    \end{minipage}%
\end{figure}

\begin{lemma}\label{lemma02.6}
Let  $n$ and $k$ be integers such that $k\ge 2$ and $n\geq 3k+1$. Then
   \begin{equation} \label{eq2.5}
  ex_{ori}(n,\overrightarrow{S_{k,1}})\geq \left\lfloor\frac{(n+k-1)^2}{4}\right\rfloor.
  \end{equation}
\end{lemma}

\begin{proof}
  Let $D_k$ be an oriented graph with $n$ vertices and $V(D_k)$ has a bipartition $V(D_k)= X\cup Y$ such that
  \begin{itemize}
  \item[(i)]\vskip -6pt$|X|\in \left\{\lfloor (n-k+1)/2\rfloor, \lceil (n-k+1)/2\rceil\right\}$;
  \item[(ii)]\vskip -6pt $xy\in A(D)$ for all $x\in X,y\in Y$;
  \item[(iii)]\vskip -6pt $D[X]$ is an empty graph;
  \item[(iv)]\vskip -6pt $D[Y]$ is an in-degree $(k-1)$-regular oriented graph.
  \end{itemize}

We remark that the above $D_k$ always exists. Since $n\geq 3k+1$, we have $|Y|\ge 2k-1$.  Denote by $Y=\{y_1,y_2,\ldots ,y_t\}$.  Then
  $y_iy_j\in A(D)$ if and only if $i-j~({\rm mod}~ t)\in  \{1,2,\ldots,k-1\}$ guarantees an  in-degree $(k-1)$-regular oriented graph $D[Y]$.

  Obviously, $D_k$ is an $\overrightarrow{S_{k,1}}$-free oriented graph with $\left\lfloor{(n+k-1)^2}/{4}\right\rfloor$ arcs.
  Therefore, we have \eqref{eq2.5}.
\end{proof}

\section{Proofs of the main results}
In this section, we present the proofs of our main results.

\subsection{Proof of Theorem \ref{th1.1}}

 Suppose $D\in EX_{ori}(n,\overrightarrow{S_{2,1}})$ with $n\ge 16$. Partition $V(D)=X\cup Y$ with $$X=\{x\in V(D): d^-(x)\leq 1\}, \quad Y=\{y\in V(D): d^-(y)\geq 2\}.$$
We have the following claim.\\

 \begin{claim}\label{claim2.1}
For the above $X$ and $Y$, we have
 \begin{equation}\label{eq04.1}
 d_{Y}^{-}(y)\leq 1\quad {\rm for~~all\quad} y\in Y,
 \end{equation}
 and
 \begin{equation}\label{eq02.1}
 \frac{n}{2}-\sqrt{\frac{n}{2}}<|X|< \frac{n}{2}+\sqrt{\frac{n}{2}}.
 \end{equation}
 \end{claim}

{\it Proof of Claim \ref{claim2.1}.}
By Lemma \ref{lemma2.3}, we have (\ref{eq04.1}).
Since $D \in EX_{ori}(n,\overrightarrow{S_{2,1}})$, by Lemma \ref{lemma02.6}
we have
\begin{equation*}
  \begin{split}
  \frac{n^2+2n}{4}<\left\lfloor \frac{(n+1)^2}{4}\right\rfloor \leq a(D)&\leq \sum_{v\in X}d_X^-(v)+\sum_{v\in Y}d_Y^-(v)\\
  &\leq |X|+(|X|+1)(n-|X|)=-|X|^2+n|X|+n,
\end{split}
\end{equation*}
which leads to \eqref{eq02.1}.\\

By Claim \ref{claim2.1}, we have $\Delta^- \le |X|+1$. Denote by $s=|X|$. We distinguish three cases.

{\it Case 1.} $\Delta^- \leq s-1$. We have
\begin{equation}
\begin{split}
    a(D)&=\sum_{x\in X}d^-(x)+\sum_{y\in Y}d^-(y)\leq s+(s-1)(n-s)\\
    &=-s^2+(n+2)s-n\leq \frac{n^2+4}{4}<\left\lfloor\frac{(n+1)^2}{4}\right\rfloor.
\end{split}
\end{equation}

{\it Case 2.} $\Delta^- =s$. Then we have
\begin{equation}\label{eq4.7}
\begin{split}
    a(D)&=\sum_{x\in X}d^-(x)+\sum_{y\in Y}d^-(y)
    \le s+(n-s)s \le \left\lfloor\frac{(n+1)^2}{4}\right\rfloor.
\end{split}
\end{equation}
Suppose equality in \eqref{eq4.7} holds. Then we have
  $s\in\{\lfloor (n+1)/2\rfloor,\lceil (n+1)/2\rceil\}$ and
\begin{equation}\label{eq4.8}
d^-(x)=1,~~d^-(y)=s\quad {\rm for~~all}\quad x\in X,~~y\in Y.
\end{equation}
If $D[Y]$ is an empty graph. Then $$N^-(v)=X\quad {\rm for~~ all}\quad  v\in Y,$$ which implies $a(X)=s\ge 6$. Therefore, $D[X]$ contains a 2-matching, which leads to an $\overrightarrow{S_{2,1}}$ in $D$, a contradiction.
If $D[Y]$ is not an empty graph, choose an arc $y_1y_2\in D[Y]$. Then by \eqref{eq4.8} we have
\begin{equation}\label{eq4.9}
|N^-_X(y_1)\cap N^-_X(y_2)|\geq s-2> 3.
\end{equation}
Take a vertex $x\in N^-_X(y_1)\cap N^-_X(y_2)$ with its unique in-neighbor  $w$. Then by \eqref{eq4.9} there exists a vertex $x'\in N^-_X(y_1)\cap N^-_X(y_2)-\{x,w\}$. It follows that
$D$ has an $\overrightarrow{S_{2,1}}=\{wxy_2,x'y_1y_2\}$, a contradiction.

Therefore,  we have  $a(D)< \left\lfloor (n+1)^2 / 4 \right\rfloor $ in this case.

{\it Case 3.} $\Delta^- = s+1$. Let $v$ be a vertex with maximum in-degree. Then by \eqref{eq04.1} we have  $N^-(v)=X\cup \{y_1\}$ for some $y_1\in Y$.
We distinguish two subcases.

{\it Subcase 3.1.} $D[X]$ is not an empty graph.
We choose an arbitrary arc $x_1x_2\in A(D)$ with $x_1,x_2\in X$.
It follows that $N^-(y_1)\subseteq \{x_1,x_2\}$; otherwise a vertex $w\in N^-(y_1)-\{x_1,x_2\}$ leads to an $\overrightarrow{S_{2,1}}=\{x_1x_2v,wy_1v\}$, a contradiction. Since $d^-(y_1)\ge 2$, we have  $N^-(y_1)= \{x_1,x_2\} $, which implies that $a(X)=1$, as
  $x_1x_2$ is chosen arbitrarily in $D[X]$.

Moreover, we have $a(y_1,X)=0$; otherwise if $y_1x_3\in A(D)$ for some $x_3\in X$, then $x_3\not\in \{x_1,x_2\}$ and $D$ contains  an
$\overrightarrow{S_{2,1}}=\{x_1  x_2  v, y_1x_3v\}$, a contradiction.

 Therefore, we have
\begin{equation*}
\begin{split}
    a(D)&=a(X)+a(X,Y)+a(Y,X)+a(Y)\\
    &=a(X)+[a(X,Y\setminus y_1)+a(Y\setminus y_1,X)]+a(Y)+[a(y_1,X)+a(X,y_1)]\\
    &\leq 1+(n-s-1)s+(n-s)+2< \left\lfloor\frac{(n+1)^2}{4}\right\rfloor.
\end{split}
\end{equation*}

{\it Subcase 3.2.} $D[X]$ is  an empty graph.
We  count $a(D)$ as follows:
\begin{equation}\label{eq03.8}
\begin{split}
    a(D)&=a(X)+a(X,Y)+a(Y,X)+a(Y)\\
    &\leq s(n-s)+(n-s)\leq \left\lfloor\frac{(n+1)^2}{4}\right\rfloor.
\end{split}
\end{equation}
Suppose equality in \eqref{eq03.8} holds. Then $s\in \left\{\lfloor (n-1)/2\rfloor, \lceil (n-1)/2\rceil\right\}$,
\begin{equation}\label{eq03.2}
d^-_Y(y)=1   \quad {\rm for~~all}\quad  y\in Y
\end{equation}
and
\begin{equation}\label{eq03.3}
 \quad \mid \{xy,yx\}\cap A(D)\mid=1 \quad {\rm for~~all}\quad x\in X, ~y\in Y.
\end{equation}
Next we prove
\begin{equation}\label{eq03.9}
[Y,X]= \emptyset.
\end{equation}

To the contrary, suppose $[Y,X]\neq \emptyset$. Denote by $X_1=\{v\in X: ~d^-(x)=1\}$. Then we have $X_1\ne \emptyset$ and $N^-(X_1)\subseteq Y$. Moreover, all vertices in $X_1$ has the same in-neighborhood, say, $\{y_1\}$; otherwise $y_1x_1, y_2x_2\in A(D)$ with $x_1,x_2\in X_1$, $y_1,y_2\in Y$ leads to an $\overrightarrow{S_{2,1}}=\{y_1x_1v,y_2x_2v\}$, a contradiction. Therefore, by \eqref{eq03.3}, we have
$$X\subset N^-(y)\quad{\rm for~~all}\quad y\in Y\setminus\{y_1\}.$$
On the other hand, by \eqref{eq03.2}, $D[Y]$ has an arc $y_2y_3$ with $y_2,y_3\in Y\setminus \{y_1\}$.
Take two vertices $x_1\in X_1$ and $x_2\in X\setminus\{x_1\}$. Then $D$ contains an $\overrightarrow{S_{2,2}}=\{y_1x_1y_3,x_2y_2y_3\}$, a contradiction.  Therefore, we have \eqref{eq03.9}.\\

Combining all the above cases, we have $a(D)\le \left\lfloor (n+1)^2/4 \right\rfloor$, with equality only if  $V(D)$ has a  partition $V(D)=X\cup Y$ such that the conditions (i), (ii), (iii) in Theorem \ref{th1.1} hold.

On the other hand,  let $D$ be an oriented graph whose vertex set $V(D)$ can be partitioned as $V(D)=X\cup Y$ such that the conditions (i), (ii), (iii) in Theorem  \ref{th1.1} hold. Then $D$  is $\overrightarrow{S_{2,1}}$-free with $a(D)=\lfloor(n+1)^2/4\rfloor.$

This completes the proof of Theorem \ref{th1.1}.\qed

\subsection{Proof of Theorem \ref{th1.2}}

Suppose $D\in EX_{ori}(n,\overrightarrow{S_{3,1}})$ with $n\ge 40$. We   partition $V(D)=X\cup Y$ with  $$ X=\{u: d^-(u)\leq 2\},\quad Y=\{u: d^-(u)\geq 3\} .$$
Let $s=|X|$.  We need  the following claim.\\

\begin{claim}\label{claim3.1}
For the above $X$ and $Y$, we have 
 \begin{equation}\label{eq2.8}
 d^-_{Y}(v)\leq 2   \quad{\rm for~~ all }\quad v\in V(D)
 \end{equation}
and
   \begin{equation}\label{eq2.9}
   \frac{n}{2}-\sqrt{n}<  |X|< \frac{n}{2}+\sqrt{n}.
   \end{equation}
 \end{claim}

 {\it Proof of Claim \ref{claim3.1}.} For a vertex $w\in V(D)$  with the maximum in-degree, we assert that
\begin{equation}\label{EQ3.16}
\left|\{u\in N^-(w):~d^-(u)\geq 3\}\right|\leq 2.
 \end{equation}
 To the contrary, suppose $|\{u\in N^-(w):d^-(u)\geq 3\}|\geq 3$.
Applying Lemma \ref{lemma2.5}, we have $d^-(N^-(v))=(3,3,3,1,0,0,\ldots ,0)$ or $(3,3,3,0,0,\ldots ,0)$.
 It follows that
  \begin{equation*}
  \begin{split}
    a(D)&\leq \sum_{u\in N^-(w)}d^-(u)+\sum_{u\notin N^-(w)}d^-(u) \leq 10+\sum_{u\notin N^-(w)}d^-(u) \\
    &\leq 10+\Delta^- (n-\Delta^- )  < \left\lfloor \frac{(n+2)^2}{4}\right\rfloor,
  \end{split}
  \end{equation*}
which contradicts \eqref{eq2.5}.

 By  \eqref{EQ3.16}, we have $$|X|\ge \Delta^- -2 \quad{\rm and }\quad |Y|\le n-\Delta^{-}(D)+2 .$$

Suppose  $D$ has  a vertex $v$ such that $d_Y^-(v)\ge 3$, i.e.,   $N^-(v)$ contains $3$ vertices with in-degree $\ge 3$. By \eqref{EQ3.16}, we have $d^-(v)\equiv r<\Delta^- $.
Applying Lemma \ref{lemma2.5}, we get $d^-(N^-(v))=(3,3,3,1,0,\ldots,0)$ or $(3,3,3,0,\ldots,0)$.
It follows that
$$|X-N^-(v)|\ge \Delta^- -2-(r-3)= \Delta^-  - r+1\quad {\rm and}\quad  |Y-N^-(v)|\le n-\Delta^{-}-1 . $$
Therefore, we have
\begin{equation*}
  \begin{split}
    a(D)&=\sum_{u\in X-N^-(v)}d^-(u)+\sum_{u\in Y-N^-(v)}d^-(u)+\sum_{u\in N^-(v)}d^-(u)\\
    &\leq  2(\Delta^- -r+1)+\Delta^- (n-\Delta^{-}-1)+10\\
    &\leq \frac{n^2+2n+49-2r}{4}< \left\lfloor\frac{(n+2)^2}{4}\right\rfloor,
   \end{split}
  \end{equation*}
which contradicts \eqref{eq2.5}. Hence, we have \eqref{eq2.8}.

Since $D \in EX_{ori}(n,\overrightarrow{S_{3,1}})$, by Lemma \ref{lemma02.6}
  we have
\begin{equation*}
  \begin{split}
  \frac{n^2+4n}{4}<\left\lfloor \frac{(n+2)^2}{4}\right\rfloor \leq a(D)&\leq \sum_{v\in X}d_X^-(v)+\sum_{v\in Y}d_Y^-(v)\\
  &\leq 2|X|+(|X|+2)(n-|X|)=-|X|^2+n|X|+2n,
\end{split}
\end{equation*}
which leads to \eqref{eq2.9}.

This completes the proof of Claim \ref{claim3.1}. \\

Suppose $D[X]$ is empty. Notice that $a(X,Y)+a(Y,X)\le s(n-s)$. We have the following bound on the size of $D$:
\begin{equation}\label{eq03.15}
  \begin{split}
  a(D)&\leq a(X)+a(X,Y)+a(Y,X)+a(Y)\leq s(n-s)+2(n-s)\leq \left\lfloor (n+2)^2/4\right\rfloor.
\end{split}
\end{equation}

Suppose   equality in \eqref{eq03.15} holds, then  $s\in \left\{\lfloor n/2\rfloor-1, \lceil n/2\rceil-1\right\}$,
\begin{equation}\label{eq3.2}
d^-_Y(y)=2 \quad  {\rm and}\quad \mid \{xy,yx\}\cap A(D)\mid=1 \quad {\rm for~~all}\quad x\in X, ~y\in Y.
\end{equation}
Next we prove $[Y,X]= \emptyset$.\\

 To the contrary, suppose $[Y,X]\neq \emptyset$. We call $y\in Y$ a {\it bad vertex} if $d^-_X(y)=1,d^+_X(y)=s-1$.
Then  $Y$ contains at most two bad vertices. In fact, if $Y$ has three bad vertices $v_1, v_2,v_3$, then by the definition of bad vertices we have $d^+_X(v_i)=s-1$ for $i=1,2,3$, which leads to $$|N^+_X(v_1)\cap N^+_X(v_2) \cap N^+_X(v_3)|\geq s-3\geq 1.$$
It follows that there is vertex $u\in X$ with $\{v_1,v_2,v_3\}\subseteq N^-(u)$, which contradicts the construction of $X$.

Suppose $Y$ has two bad vertices $v_1,v_2$.
Then $|N^+_X(v_1)\cap N^+_X(v_2)|\geq s-2 \geq 1$.
Take $x'\in N^+_X(v_1)\cap N^+_X(v_2)$ and
 $y'\in \{v_1,v_2\}$. Then $y'x'\in [Y,X]$. Now we prove
 \begin{equation}\label{eq3.3}
 N^+(x')\subseteq  N^+(y').
 \end{equation}

Take  any $y_0\in N^+(x')$.  Since $D[X]$ is empty, we have $y_0\in Y$. By \eqref{eq3.2} we may  denote $N^-_Y(y_0)=\{y_1,y_2\}$. Since $d^-(y_i)\geq 3$,  we have $N^-_X(y_i) \neq \emptyset$ for $i=1,2$.
We assert that
\begin{equation}\label{eq3.4}
N_X^-(y_i)\setminus \{x'\}\ne \emptyset \quad {\rm for}\quad i=1,2.
\end{equation}
In fact, if $N_X^-(y_i)=\{x'\}$ for $i\in \{1,2\}$,  then by \eqref{eq3.2} we have
$N^+_X(y_i)=X\setminus \{x'\}$, which means $y_i$ is a bad vertex, i.e., $y_i\in \{v_1,v_2\}$. It follows that $y_ix'\in A(D)$, which contradicts $N^-_X(y_i)=\{x'\}$.

 Since $d^-_Y(y_1)=d^-_Y(y_2)=2$, $D$ has a matching $M_1=\{y_3y_1,y_4 y_2\}$ with $y_1,y_2,y_3,y_4\in Y$.  Take $x_i\in N_X^-(y_i)\setminus \{x'\}$ for $i=1,2$, where $x_1$ and $x_2$ may be identical.
Then $M_2=\{x_1 y_1,y_4 y_2\},M_3=\{y_3y_1,x_2y_2\}$ are matchings in $D$.

Since $D$ is $\overrightarrow{S_{3,1}}$-free,  for every matching $M=\{w_1y_1,w_2y_2\} $ in $D$
 with $x'\not\in\{w_1,w_2\}$, we have $y'\in V(M)=\{w_1,w_2,y_1,y_2\}$.
Therefore,
 $$y'\in V(M_1)\cap V(M_2)\cap V(M_3)=\{y_1,y_2\}=N^-_Y(y_0).$$
Hence, we get $y_0\in N^+(y')$ and  \eqref{eq3.3} holds.

Since $d^-(x')\leq 2$ and $d^-_Y(x')+d^+_Y(x')=n-s$, we have $d^+_Y(x')\geq n-s-2$. Notice that \eqref{eq3.3} implies $d^+_Y(y')\ge d^+_Y(x')$. Since $D$ is an oriented graph and $d^-_Y(y')+d^+_Y(y')\leq |Y|-1=n-s-1$, we have $d^-_Y(y')\leq 1$, which contradicts
 $d^-_Y(y)=2$.

If $Y$ contains one    bad  vertex $v$, take
$y'=v$ and $x'\in N^+(v)$; if $Y$ contains no bad vertex, choose any $x'\in X$ and $y'\in Y$ with $y'x'\in A(D)$.  Applying the same arguments as above we can deduce a contradiction, noticing that \eqref{eq3.4} is obvious when $Y$ contains no bad vertex.

Therefore, we have $$[Y,X]= \emptyset\quad {\rm and}\quad  [X,Y]=\{uv, u\in X, v\in Y\}.$$ Hence, $D$ is an oriented graph satisfying the conditions (i), (ii) and (iii) in Theorem \ref{th1.2}. \\

On the other hand, any oriented graph $D$  satisfying the conditions (i), (ii) and (iii) in Theorem \ref{th1.2} is $\overrightarrow{S_{3,1}}$-free with $\lfloor(n+2)^2/4\rfloor$ arcs.  Therefore,  it suffices to confirm   $a(D)<\left\lfloor (n+2)^2/4\right\rfloor$ when $D[X] $ is not empty. By Claim \ref{claim3.1} we have $\Delta^- \leq s+2.$ We distinguish three cases.\\

 {\it Case 1.} $\Delta^- \leq s$. We count the size of $D$ as follows:
\begin{equation}\label{eq3.5}
\begin{split}
    a(D) \leq \sum_{x\in X}d^-(x)+\sum_{y\in Y}d^-(y)
     \leq  2s+(n-s)s \leq \left\lfloor (n+2)^2/4\right\rfloor.
\end{split}
\end{equation}

Suppose equality in \eqref{eq3.5} holds. Then $s\in \left\{\lfloor n/2\rfloor+1, \lceil n/2\rceil+1\right\}$,
\begin{equation}\label{eq3.6}
  \begin{split}
  d^-(x)=2  \quad {\rm and}\quad d^-(y)=s \quad {\rm for~~all }\quad x\in X, y\in Y.
\end{split}
\end{equation}
We assert that  $D[Y]$ is   an empty graph. To the contrary, choose any arc $y_1y_0\in A(D)$ with $y_0,y_1\in Y$.
Since $d^-(y_0)=s$ and $d^-_Y(y_0)\leq 2$, we have  $$d^-_X(y_0)\geq s-2 \quad{\rm and}\quad |X-N^-(y_0)|\leq 2.$$
If $|X\cap N^-(y_0)-N^+(y_1)|\leq 1$,
then we have
\begin{equation*}
   |X\cap N^+(y_1)| \geq |X\cap N^-(y_0)\cap N^+(y_1)|\geq |X\cap N^-(y_0)|-1
  = |X|-|X-N^-(y_0)|-1\geq s-3.
\end{equation*}
Since $D$ is an oriented graph and $d^-_Y(y_1)\leq 2$, we have
\begin{equation}
   |X\cap N^-(y_1)|\leq |X|-|X\cap N^+(y_1)|\leq 3
\end{equation}
and
\begin{equation}
  d^-(y_1)=d^-_X(y_1)+d^-_Y(y_1)=|X\cap N^-(y_1)|+d^-_Y(y_1)\leq 5
\end{equation}
which contradicts $d^-(y_1)=s \geq 14$ by \eqref{eq2.9}. Therefore, we have  $$|X\cap N^-(y_0)-N^+(y_1)|\geq 2.$$
Choose  $x_1,x_2\in X\cap N^-(y_0)-N^+(y_1)$. Then $d^-(x_1)=d^-(x_2)=2$ and
we can find a  matching $M=\{w_1 x_1,w_2x_2\}$ with $y_1\notin \{w_1,w_2\}$.
Since $D$ is $\overrightarrow{S_{3,1}}$-free, we have $N^-(y_1)\subseteq \{w_1,w_2,x_1,x_2\}$  and $d^-(y_1)\leq 4$, which contradicts $d^-(y_1)=s\geq 14$.
Therefore,  $D[Y]$ is an empty graph.

By \eqref{eq3.6}, we have
\begin{equation}\label{eq3.9}
N^-(y)=X \quad {\rm for~~ all}\quad y\in Y,
\end{equation} which leads to $a(Y,X)=0$.
Hence, given any 2-matching $M$ in $D[X]$, we have
\begin{equation}\label{eq3.10}
N^-(v)\subset V(M) \quad{\rm for~~ all}\quad v\in X\setminus V(M),
\end{equation}
since otherwise we can find an $\overrightarrow{S_{3,1}}$ in $D$.

By \eqref{eq3.6} and \eqref{eq3.9}, $D$ contains a matching $M_0=\{x_1  x_2,x_3 x_4\}$ with $\{x_1,x_2,x_3,x_4\}\subset X$.
Since $d^-(x_1)=d^-(x_3)=2$, there exists a vertex $x_5\in X\setminus \{x_1,x_2,x_3,x_4\}$ such that either  $x_5x_1\in A(D)$  or $x_5x_3\in A(D)$.  Without loss of generality, we assume
$x_5x_3\in A(D)$.

By \eqref{eq3.6} and \eqref{eq3.10}, we have $|N^-(x_5)|=2$ and $N^-(x_5)\subset \{x_1,x_2,x_3,x_4\}$. We can deduce  contradictions for all the six choices of $N^-(x_5)$.
We take  $N^-(x_5)=\{x_1,x_2\}$ for example, as the argument for other cases are similar. Note that  $D$ has four matchings $M_1=\{x_1 x_2,x_3 x_4\},M_2=\{x_1 x_2,x_5 x_3\},M_3=\{x_1 x_5,x_3 x_4\},M_4=\{x_2 x_5,x_3 x_4\}$. By \eqref{eq3.10}, we have $N^-(u)\subset V(M_1)\cap V(M_2)\cap V(M_3)\cap V(M_4)=\{x_3\}$ for all $u\in X\setminus\{x_1,x_2,x_3,x_4,x_5\},$ which contradicts  \eqref{eq3.6} as $|X|>5$.

 Therefore, we have  $a(D) <\left\lfloor (n+2)^2/4\right\rfloor $  in this  case.\\

 {\it Case 2.} $\Delta^- =s+2$.
Choose a vertex  $y_0\in Y$ with maximum in-degree. By \eqref{eq2.8} we have   $N^-(y_0)=X\cup \{y_1,y_2\}$ form some $y_1,y_2\in Y$. Recall that $d^-(y_i)\ge 3$ for $i=1,2$. We distinguish three subcases.

 {\it Case 2.1.} $d^-(y_1)\geq 4$ and $d^-(y_2)\geq 4$.
Since  $D[X]$ is not an empty graph, we choose an arc $x_1x_2\in A(D)$ with $x_1,x_2\in X$.
Then  we have  $\widetilde{d}^-(y_1,y_2, x_2)\rhd (3,2,1)$.
Applying  Lemma \ref{lemma2.2},
$D$ contains a $\overrightarrow{S_{3,1}}$   centered at $y_0$, a contradiction.

 {\it Case 2.2.} One of $\{y_1,y_2\}$ has in-degree 3, and the other one has in-degree larger than 3, say, $d^-(y_1)=3$ and $d^-(y_2)\geq 4$.
If there is a vertex $ x\in X$ such that $ N^-(x)-\{y_1,y_2\}\neq \emptyset$ and $xy_1\notin A(D)$,
then $\widetilde{d}^-(y_1,y_2, x)\rhd (3,2,1)$.
Applying  Lemma \ref{lemma2.2},
$D$ contains an $\overrightarrow{S_{3,1}}$   centered at $y_0$, a contradiction. Therefore, for any $x\in X$, we have either
\begin{equation}\label{eq3.12}
N^-(x)\subseteq \{y_1,y_2\}
\end{equation}
or
\begin{equation}\label{eq3.13}
  xy_1\in A(D).
  \end{equation}

 Since $d^-(y_1)=3$,  $X$  has at most 3 vertices satisfying \eqref{eq3.13}. Hence, $X$ has at most 3 vertices  such that \eqref{eq3.12} does not hold, which leads to $a(X)\leq 6$.

 Since $D[X]$ is not empty, there is an arc $x_1x_2\in A(D)$ with $x_1,x_2\in X$. Since $d^-(y_2)\ge 4$, there is a vertex $w_1\in N^-(y_2)- \{x_1,x_2\}$. It follows that
 \begin{equation}\label{eq3.14}
 a(y_1,X)+a(X,y_1)\le 3;
  \end{equation}
  otherwise, there is a vertex $x_3\in N^+_X(y_1)\cup N^-_X(y_1)-\{x_1,x_2,w_1\}$, which leads to an $\overrightarrow{S_{3,1}}$ centered at $y_0$ as the following figures shows, a contradiction.
\begin{figure}[h]
    \centering
    \begin{minipage}{0.5\textwidth}
    \centering
    \begin{tikzpicture}
      \filldraw (0,1) circle (0.1);
      \filldraw (0,0) circle (0.1);
      \filldraw (0,-1) circle (0.1);
      \filldraw (2,0) circle (0.1);
      \filldraw (1,1) circle (0.1);
      \filldraw (1,0) circle (0.1);
      \filldraw (1,-1) circle (0.1);
      \draw[-latex] (0,1)--(0.9,1);
      \draw[-latex] (0,0)--(0.9,0);
      \draw[-latex] (0,-1)--(0.9,-1);
      \draw[-latex] (1,1)--(1.9,0.1);
      \draw[-latex] (1,0)--(1.9,0);
      \draw[-latex] (1,-1)--(1.9,-0.1);
      \node[left] at (0,1) {$x_1$};
      \node[above] at (1,1) {$x_2$};
      \node[left] at (0,0) {$w_1$};
      \node[above] at (1,0) {$y_2$};
      \node[left] at (0,-1) {$x_3$};
      \node[above] at (1,-1) {$y_1$};
      \node[right] at (2,0) {$y_0$};
    \end{tikzpicture}
    \caption{$x_3\in N^-( y_1)$}
    \label{fig:}
    \end{minipage}%
    \begin{minipage}{0.5\textwidth}
    \centering
    \begin{tikzpicture}
      \filldraw (0,1) circle (0.1);
      \filldraw (0,0) circle (0.1);
      \filldraw (0,-1) circle (0.1);
      \filldraw (2,0) circle (0.1);
      \filldraw (1,1) circle (0.1);
      \filldraw (1,0) circle (0.1);
      \filldraw (1,-1) circle (0.1);
      \draw[-latex] (0,1)--(0.9,1);
      \draw[-latex] (0,0)--(0.9,0);
      \draw[-latex] (1,-1)--(0.1,-1);
      \draw[-latex] (1,1)--(1.9,0.1);
      \draw[-latex] (1,0)--(1.9,0);
      \draw[-latex] (0,-1)--(1.9,-0.1);
      \node[left] at (0,1) {$x_1$};
      \node[above] at (1,1) {$x_2$};
      \node[left] at (0,0) {$w_1$};
      \node[above] at (1,0) {$y_2$};
      \node[left] at (0,-1) {$x_3$};
      \node[above] at (1,-1) {$y_1$};
      \node[right] at (2,0) {$y_0$};
    \end{tikzpicture}
    \caption{$ x_3\in N^+( y_1)$}
    \label{fig:}
    \end{minipage}
\end{figure}

Therefore,   we can count $a(D)$ as  follows:
\begin{equation}\label{eq3.15}
\begin{split}
    a(D)&=a(X)+a(X,Y)+a(Y,X)+a(Y)\\
    &=a(X)+[a(X,Y\setminus y_1)+a(Y\setminus y_1,X)]+[a(X,y_1)+a(y_1,X)]+a(Y)\\
    &\leq 6+(n-s-1)s+3+2(n-s)\\
    & <\left\lfloor\frac{(n+2)^2}{4} \right\rfloor.
\end{split}
\end{equation}

 {\it Case 2.3.} $d^-(y_1)=d^-(y_2)=3$.
If there is a vertex $ x\in X$ such that $x\not\in  N^-(y_1)\cup N^-(y_2)$ and  $ N^-(x)-\{y_1,y_2\}\neq \emptyset$,
then $\widetilde{d}^-(y_1,y_2, x)\rhd (3,2,1)$.
Applying  Lemma \ref{lemma2.2},
$D$ contains a $\overrightarrow{S_{3,1}}$   centered at $y_0$, a contradiction.
Therefore, for any $ x\in X$, we have either \eqref{eq3.12} or
\begin{equation}\label{eq3.016}
x \in  N^-(y_1)\cup N^-(y_2).
\end{equation}

Since $|N^-(y_1)\cup N^-(y_2)|\leq 6$, $X$ has at most 6 vertices
that do not satisfying \eqref{eq3.12}, which leads to $a(X)\leq 12$.

Again, since $D[X]$ is not empty, there is an arc $x_1x_2\in a(D)$ with $x_1,x_2\in X$.
Notice that $d^-(y_1)=d^-(y_2)=3$ and $D$ is an oriented graph, we have $$N^-(y_1)-\{x_1,x_2,y_2\}\neq \emptyset  \quad {\rm or}\quad   N^-(y_2)-\{x_1,x_2,y_1\}\neq \emptyset.$$
Without loss of generality, we assume  $N^-(y_2)-\{x_1,x_2, y_1\}\neq \emptyset$ contains a vertex $w_1$. Then similarly as in Case 2.2, we have \eqref{eq3.14}.

 Now using the same calculation as in \eqref{eq3.15} we have $a(D) <\left\lfloor (n+2)^2/4\right\rfloor $.\\

 {\it Case 3.} $\Delta^- =s+1$. We distinguish two subcases.

 {\it Subcase 3.1.} There exists a vertex $y_0\in Y$ with in-degree $\Delta^- $ such that   $d^-_X(y_0)=s-1$ and  $d^-_Y(y_0)=2$. Let $N^-_Y(y_0)=\{y_1,y_2\}$ and $X-N^-_X(y_0)=\{x_0\}$. Denote by $X'=N^-_X(y_0)$. We distinguish three subcases.

  {\it Subcase 3.1.1.} $d^-(y_1)\geq 4,d^-(y_2)\geq 4$. If there is a vertex  $ x \in X' $ such that $N^-(x)-\{y_1,y_2\}\neq \emptyset$, then $\widetilde{d}^-(y_1,y_2, x)\rhd (3,2,1)$.
Applying Lemma \ref{lemma2.2}, $D$ contains an $\overrightarrow{S_{3,1}}$,  a contradiction. Therefore, we have
 \begin{equation}\label{eq3.17}
 N^-(x)\subseteq \{y_1,y_2\} \quad{\rm for ~~all}\quad  x\in X',
 \end{equation}
 which leads to
$$\sum_{x\in X'}d^-(x)\leq d^+_{X'}(y_1)+d^+_{X'}(y_2).$$
Notices that
$$d^+_{X'}(y_i)+d^-(y_i)\le d^+_X(y_i)+d^-_X(y_i)+d^-_Y(y_i)\le s+2\quad {\rm for}\quad i=1,2.$$
We have
\begin{equation}\label{eq3.16}
\begin{split}
    a(D)&=d^-(x_0)+\sum_{x\in X'}d^-(x)+d^-(y_1)+d^-(y_2)+\sum_{y\in Y\setminus \{y_1,y_2\}}d^-(y)\\
    &\leq d^-(x_0)+[d^+_{X'}(y_1)+d^-(y_1)]+[d^+_{X'}(y_2)+d^-(y_2)]+\sum_{y\in Y\setminus \{y_1,y_2\}}d^-(y)\\
    &\leq 2+2(s+2)+(s+1)(n-s-2)\\
    &<  \left\lfloor\frac{(n+2)^2}{4} \right\rfloor.
\end{split}
\end{equation}

 {\it Subcase 3.1.2.}~~ One of $\{y_1,y_2\}$ has in-degree 3, and the other one has in-degree larger than 3, say, $d^-(y_1)=3$ and $d^-(y_2)\geq 4$. Using the same arguments as in Case 2.2, for every $x\in X'$, we have either \eqref{eq3.12} or  \eqref{eq3.13}.  Since $d^-(y_1)=3$,  $X'$ has at most 3 vertices  such that \eqref{eq3.12} does not hold, which leads to $\sum_{x\in X'}d^-_X(x)\leq 6$. It follows that $$a(X)\le\sum_{x\in X'}d^-_X(x)+d^-(x_0)\leq 8.$$

 If \eqref{eq3.12} holds for all $x\in X'$, using the same calculation as in Subcase 3.1.1 we have  $a(D)<\lfloor (n+2)^2/4\rfloor.$
If  there is a vertex $x_1\in X'$ such that $N^-(x_1)\not\subseteq \{y_1,y_2\}$, i.e., $x_1$ has an in-neighbor $w_1\in V(D)\setminus\{y_1,y_2\}$, then $d^-(y_2)\ge 4$ implies that $y_2$ has an in-neighbor $w_2$ from $V(D)\setminus\{x_1,y_1,w_1\}$. It follows that
\begin{equation}\label{eq03.32}
a(X,y_1)+a(y_1,X)\le 4;
\end{equation}
otherwise there is a vertex $x_2\in X'-\{x_1,w_1,w_2\}$ such that either $x_2y_1\in A(D)$ or $y_1x_2\in A(D)$, which both guarantee  an $\overrightarrow{S_{3,1}}$ in $D$, a contradiction.
Therefore, similarly as in \eqref{eq3.15}, we have
\begin{equation*}
\begin{split}
    a(D)
    &=a(X)+[a(X,Y\setminus y_1)+a(Y\setminus y_1,X)]+[a(X,y_1)+a(y_1,X)]+a(Y)\\
    &\leq 8+s(n-s-1)+4+2(n-s) \\
    &<  \left\lfloor\frac{(n+2)^2}{4} \right\rfloor.
\end{split}
\end{equation*}

{\it Subcase 3.1.3.} $d^-(y_1)=d^-(y_2)=3$. Then similarly as in Subcase 2.3, for any
$  x\in X'$, we have either \eqref{eq3.12} or \eqref{eq3.016}.
If \eqref{eq3.12} holds for all $x\in X'$, then
we have \eqref{eq3.16} again.

If  there is a vertex $x_1\in X'$ such that $N^-(x_1)\not\subseteq \{y_1,y_2\}$, i.e., there is a vertex $x_1\in X'$ which has an in-neighbor $w_1\in V(D)\setminus\{y_1,y_2\}$, then $d^-(y_1)=d^-(y_2)=3$ implies that at least one of  $\{y_1,y_2\}$, say $y_2$, has an in-neighbor $w_2$ from $V(D)\setminus\{x_1,y_1,y_2,w_1\}$. Similarly as in Subcase 3.1.2, we have \eqref{eq03.32}.
On the other hand, since $|N^-(y_1)\cup N^-(y_2)|\leq 6$, $X'$ has at most 6 vertices $x$   such that $N^-(x)\not\subseteq \{y_1,y_2\}$, which leads to
$\sum_{x\in X'}d^-_X(x)\leq 12$ and $a(X)=\sum_{x\in X}d^-_X(x)\leq 14$. Therefore, similarly as in \eqref{eq3.15}, we have
\begin{equation*}
\begin{split}
    a(D)
    &=a(X)+[a(X,Y\setminus y_1)+a(Y\setminus y_1,X)]+[a(X,y_1)+a(y_1,X)]+a(Y)\\
    &\leq  14+s(n-s-1)+4+2(n-s)\\
    &<  \left\lfloor\frac{(n+2)^2}{4} \right\rfloor.
\end{split}
\end{equation*}

We remark that we need $n\ge 40$ to guarantee the above inequality.

 {\it Subcase 3.2.}  $d^-_X(y)=s,d^-_Y(y)=1$ for all $y\in Y$ with $d^-(y)=\Delta^- .$  Let $y_0\in Y$ be a vertex with maximum in-degree and
 $N^-(y_0)=X\cup \{y_1\}$ with $y_1\in Y$.
 Let $E=\{ux\in A(D):~x\in X, u\neq y_1\}$.
 We distinguish two subcase.

{\it Subcase 3.2.1.}  $D[E]$  contains a 2-matching $M_0=\{w_1x_1,w_2x_2\}$.
Since $D$ is $\overrightarrow{S_{3,1}}$-free and $N^-(y_0)=X\cup\{y_1\}$, for every 2-matching $M=\{u_1v_1,u_2v_2\}$ with $v_1,v_2\in X$, we have
\begin{equation}\label{eq3.24}
N^-(v)\subseteq V(M)\quad {\rm for~~all}\quad v\in X\cup\{y_1\}-V(M).
\end{equation}

We claim that
\begin{equation}\label{eq3.21}
|\{u\in X:~d^-(u)=2\}|\leq 4.
\end{equation}
Otherwise, we have a vertex $x_3\in X-\{w_1,w_2,x_1,x_2\}$ with $d^-(x_3)=2$.
By \eqref{eq3.24} we have $N^-(x_3)\subseteq \{w_1,w_2,x_1,x_2\}$. For each of the six choices of $N^-(x_3)$, we can always deduce a contradiction.
We take $N^-(x_3)=\{w_1,x_1\}$ as example, as the arguments for   other cases are similar.
Note that $D$ contains three matchings $M_0=\{w_1x_1,w_2x_2\}$, $M_1=\{w_1x_3,w_2x_2\}$, $M_2=\{x_1x_3,w_2x_2\}$.
By \eqref{eq3.24}, we have $N^-(y_1)\subseteq V(M_0)\cap V(M_1)\cap V(M_2)=\{w_2,x_2\}$,
which contradicts $y_1\in Y=\{u\in V(D):~d^-(u)\geq 3\}$.

 Similarly to the argument for \eqref{eq3.14}, we have
 $a(X,y_1)+a(y_1,X)\leq 4$,~which~implies~$d^-(y_1)\leq 6.$
 Therefore,   by \eqref{eq3.21} we have
\begin{equation*}
\begin{split}
    a(D)&=\sum_{u\in X}d^-(u)+\sum_{u\in Y\setminus y_1}d^-(u)+d^-(y_1)\\
    &\leq s+4+(s+1)(n-s-1)+6 \\
    & <  \left\lfloor\frac{(n+2)^2}{4} \right\rfloor.
\end{split}
\end{equation*}

{\it Subcase 3.2.2.}~~$D[E]$ contains no 2-matching.  By the construction of $E$ we have
\begin{equation}\label{Eq4.24}
 \sum_{x\in X}d^-(x)=|E|+d^+_X(y_1).
 \end{equation}
 Since $d^-_Y(y_1)\le 2$, we have
 \begin{equation}\label{Eq4.25}
 d^+_X(y_1)+d^-(y_1)\le s+2.
 \end{equation}

 We assert that
 \begin{equation}\label{Eq4.26}
\begin{split}
    |E|+\sum_{u\in Y\setminus y_1}d^-(u)\le s-1+(s+1)(n-s-1).
\end{split}
\end{equation}
 If  $|E|\leq s-1$, the assertion is obvious. Suppose $|E|\ge s$. Since $s\ge 14$ and $D[E]$ contains no 2-matching,   the underlying graph of $D[E]$ is a star. Moreover, since  $d^-(u)\le 2$ for all $u\in X$, $E$ contains at most 2 vertices from $Y$, which leads to $|E|\le s+1$.
If $|E|=s$, then $E$ covers at least one vertex $y'\in Y\setminus y_1$, which implies $y'$ is an in-neighbor of some vertex in $X$. Therefore, we have $X\not\subset N^-(y')$ and $d^-(y')<\Delta^- $. If $|E|=s+1$, then $E$ covers   two vertices $v_1,v_2\in Y\setminus y_1$, which implies   $X\not\subset N^-(v_i)$ and $d^-(v_i)<\Delta^- $ for $i=1,2$. In both cases we  have \eqref{Eq4.26}.

By \eqref{Eq4.24},\eqref{Eq4.25} and \eqref{Eq4.26} we have
 \begin{equation*}
\begin{split}
    a(D)&=\sum_{u\in X}d^-(u)+\sum_{u\in Y}d^-(u) =|E|+\sum_{u\in Y\setminus y_1}d^-(u)+d^+_X(y_1)+d^-(y_1)\\
    &\le s-1+(s+1)(n-s-1)+s+2 <\left\lfloor\frac{(n+2)^2}{4} \right\rfloor.
\end{split}
\end{equation*}

This completes the proof of Theorem \ref{th1.2}. \qed

\subsection{Proof of Theorem \ref{th1.3}}
By Lemma \ref{lemma02.6} we get the lower bound. We prove the upper bound.
Suppose $D$ is an $\overrightarrow{S_{k,1}}$-free oriented graph of order at least $3k+1$ with $k\ge 4$. We assert that
\begin{equation}\label{eq4.1}
|\{u\in N^-(v):~d^-(u)\ge 2k-1\}|< k\quad {\rm for~~all}\quad v ~~{\rm with}~~d^-(v)\ge k.
\end{equation}
To the contrary,  suppose $v$ is a vertex with in-degree at least $k$ and $S=\{u_1,\ldots,u_k\}\subseteq \{u\in N^-(v):~d^-(u)\ge 2k-1\}$.  Then $d^-(S)\rhd (2k-1,2k-1,\ldots,2k-1)$. Since $D$ is an oriented graph, we have
\begin{equation*}
d_{S}^-(u_i)\leq k-1 \quad{\rm  for \quad} i=1,\ldots,k,
 \end{equation*}
which leads to   $\widetilde{d}^-(u_1,u_2,\ldots,u_k)\rhd (k,k-1,\ldots,1)$.
 Applying Lemma \ref{lemma2.2} we get \eqref{eq4.1}.

Now we partition $V(D)$ as $V(D)=X\cup Y$ with $X=\{u:~d^-(u)\leq 2k-2\}, Y=\{u:~d^-(u)\geq 2k-1\}$. Let $s=|X|$. Then by \eqref{eq4.1} we have
$$d^-_Y(u)\le k-1\quad {\rm for~~all}\quad u\in Y,$$
which leads to
$$d^-(u)\le s+k-1\quad {\rm for~~all}\quad u\in Y.$$
It follows that
\begin{eqnarray*}
a(D)=\sum_{x\in X}d^-(x)+\sum_{y\in Y}d^-(y)\le (2k-2)s+(s+k-1)(n-s)\le \left\lfloor\frac{(n+k-1)^2}{4}\right\rfloor+(k-1)n.
\end{eqnarray*}

This completes the proof.\qed\\

{\it Remark.} It would be interesting to determine the precise value of $ex_{ori}(n,\overrightarrow{S_{k,1}})$ and the extremal digraphs for $n\ge 4$. We leave this for future work.

\section*{Acknowledgement}
 This work was supported by the National Natural Science Foundation of China (No. 12171323) and Guangdong Basic and Applied Basic Research
Foundation (No. 2022A1515011995).

\end{document}